\documentclass{amsart}
\usepackage[all]{xy}
\usepackage{verbatim}
\usepackage{color}
\usepackage{amsthm}
\usepackage{amssymb}
\usepackage[colorlinks=true]{hyperref}
\usepackage{footmisc}

\numberwithin{equation}{section}

\newtheorem{theorem}[equation]{Theorem}
\newtheorem*{theorem*}{Theorem} \newtheorem{lemma}[equation]{Lemma}

\newtheorem*{conjecture*}{Mamma Conjecture}
\newtheorem*{conjecture1*}{Mamma Conjecture (revisited)}
\newtheorem{proposition}[equation]{Proposition}
\newtheorem{corollary}[equation]{Corollary}
\newtheorem*{corollary*}{Corollary}

\theoremstyle{remark}

\newtheorem{example}[equation]{Example}

\theoremstyle{remark}
\newtheorem{remark}[equation]{Remark}

\newcommand{\cA}{{\mathcal A}}
\newcommand{\cB}{{\mathcal B}}
\newcommand{\cC}{{\mathcal C}}

\newcommand{\cF}{{\mathcal F}}

\newcommand{\cO}{{\mathcal O}}

\newcommand{\cW}{{\mathcal W}}
\newcommand{\cX}{{\mathcal X}}
\newcommand{\cY}{{\mathcal Y}}
\newcommand{\cZ}{{\mathcal Z}}

\newcommand{\bbA}{\mathbb{A}}
\newcommand{\bbB}{\mathbb{B}}
\newcommand{\bbC}{\mathbb{C}}

\newcommand{\bbP}{\mathbb{P}}

\newcommand{\bbQ}{\mathbb{Q}}
\newcommand{\bbZ}{\mathbb{Z}}

\DeclareMathOperator{\id}{id}

\DeclareMathOperator{\NChow}{NChow} 

\newcommand{\dgcat}{\mathrm{dgcat}} 

\newcommand{\perf}{\mathrm{perf}}

\newcommand{\Chow}{\mathrm{Chow}}

\newcommand{\dg}{\mathrm{dg}}

\newcommand{\Hom}{\mathrm{Hom}}

\newcommand{\op}{\mathrm{op}}

\newcommand{\too}{\longrightarrow}

\newcommand{\ie}{\textsl{i.e.}\ }

\let\oldmarginpar\marginpar
\def\marginpar#1{\oldmarginpar{\tiny #1}}

\begin{document}

\title[A note on the Schur-finiteness of linear sections]{A note on the Schur-finiteness of linear sections}
\author{Gon{\c c}alo~Tabuada}

\address{Gon{\c c}alo Tabuada, Department of Mathematics, MIT, Cambridge, MA 02139, USA}
\email{tabuada@math.mit.edu}
\urladdr{http://math.mit.edu/~tabuada}
\thanks{The author was partially supported by a NSF CAREER Award}

\date{\today}
\abstract{Making use of the recent theory of noncommutative motives, we prove that Schur-finiteness in the setting of Voevodsky's mixed motives is invariant under homological projective duality. As an application, we show that the mixed motives of smooth linear sections of certain (Lagrangian) Grassmannians, spinor varieties, and determinantal varieties, are Schur-finite. Finally, we upgrade our applications from Schur-finiteness to Kimura-finiteness.}}

\maketitle
\vskip-\baselineskip
\vskip-\baselineskip


\section{Introduction and statement of results}
Let $(\cC,\otimes, {\bf 1})$ be a $\bbQ$-linear, idempotent complete, symmetric monoidal category. Given a partition $\lambda$ of an integer $n\geq 1$, consider the corresponding irreducible $\bbQ$-linear representation $V_\lambda$ of the symmetric group $\mathfrak{S}_n$ and the associated idempotent $e_\lambda\in \bbQ[\mathfrak{S}_n]$. Under these notations, the Schur-functor $S_\lambda\colon \cC \to \cC$ sends an object $a$ to the direct summand of $a^{\otimes n}$ determined by $e_\lambda$. Following Deligne \cite[\S1]{Deligne}, an object $a \in \cC$ is called {\em Schur-finite} if it is annihilated by some Schur-functor. Among other properties, Schur-finiteness is stable under direct sums, direct summands, tensor products, and distinguished triangles (consult Guletskii \cite{Guletskii} and Mazza \cite{Mazza}).

Voevodsky introduced in \cite{Voevodsky} a triangulated category of geometric mixed motives $\mathrm{DM}_{\mathrm{gm}}(k)_\bbQ$ (over a perfect base field $k$). By construction, this category is $\bbQ$-linear, idempotent complete, symmetric monoidal, and comes equipped with a symmetric monoidal functor $M(-)_\bbQ \colon \mathrm{Sm}(k) \to \mathrm{DM}_{\mathrm{gm}}(k)_\bbQ$, defined on smooth $k$-schemes. Conjecturally, all the objects of $\mathrm{DM}_{\mathrm{gm}}(k)_\bbQ$ are Schur-finite. Thanks to the work of Kimura \cite{Kimura}, K\"unneman \cite{Kunneman}, and Shermenev \cite{Shermenev}, the (mixed) motives $M(Z)_\bbQ$ of smooth projective $k$-schemes $Z$ of dimension $\leq 1$, and of abelian varieties, are Schur-finite. Besides these cases, this important conjecture remains wide open.

Now, let $X$ be a smooth projective $k$-scheme equipped with a line bundle $\cO_X(1)$; we write $X \to \bbP(V)$ for the associated map, where $V:=H^0(X,\cO_X(1))^\ast$. Given a linear subspace $L\subset V^\ast$, consider the linear section $X_L:=X\times_{\bbP(V)}\bbP(L^\perp)$. Our motivating question in this note is the following:

\vspace{0.1cm}

{\it Question: Is the mixed motive $M(X_L)_\bbQ$ Schur-finite?}

\vspace{0.1cm}
 
As proved by Ayoub in \cite[Prop.~5.7]{Ayoub}, a positive answer to the above question in the particular case where $X$ is the projective space, would imply that all the objects of the triangulated category $\mathrm{DM}_{\mathrm{gm}}(k)_\bbQ$ are Schur-finite. This justifies the importance of linear sections in the study of the Schur-finiteness of mixed motives.

Thanks to the theory of noncommutative motives (see \S\ref{sub:homological}) and to Kuznetsov's homological projective duality (=HPD), we are now able to answer affirmatively to the aforementioned question in several cases. Assume that the category of perfect complexes $\perf(X)$ admits a Lefschetz decomposition\footnote{When $\bbA_0=\bbA_1=\cdots = \bbA_{i-1}$, the Lefschetz decomposition is called {\em rectangular}.} $\langle \bbA_0, \bbA_1(1), \ldots, \bbA_{i-1}(i-1)\rangle$ with respect to $\cO_X(1)$ in the sense of \cite[Def.~4.1]{KuznetsovHPD}. Following \cite[Def.~6.1]{KuznetsovHPD}, let $Y$ be the HP-dual of $X$, $\cO_Y(1)$ the HP-dual line bundle, and $Y\to \bbP(V^\ast)$ the map associated to $\cO_Y(1)$. Given a generic linear subspace $L \subset V^\ast$, consider the linear sections $X_L$ and $Y_L:=Y \times_{\bbP(V^\ast)} \bbP(L)$. 
\begin{theorem}[HPD-invariance\footnote{Other HPD-invariance type results were established in \cite{Crelle,periods,CD}.}]\label{thm:HPD}
Let $X$ and $Y$ be as above. Assume that $X_L$ and $Y_L$ are smooth, that $\mathrm{dim}(X_L)=\mathrm{dim}(X) -\mathrm{dim}(L)$, that $\mathrm{dim}(Y_L)=\mathrm{dim}(Y)- \mathrm{dim}(L^\perp)$, and that the category $\bbA_0$ admits a full exceptional collection. Under these assumptions\footnote{Theorem \ref{thm:HPD} holds more generally when $Y$ is singular. In this case, we need to replace $Y$ by a noncommutative resolution of singularities $\perf(Y;\cF)$; consult \cite[\S2.4]{ICM-Kuznetsov} for details.}, $M(X_L)_\bbQ$ is Schur-finite if and only if $M(Y_L)_\bbQ$ is Schur-finite.
\end{theorem}
Intuitively speaking, Theorem \ref{thm:HPD} shows that Schur-finiteness in the setting of Voevodsky's mixed motives is invariant under homological projective duality. As a consequence of this invariance, we obtain the following practical result:
\begin{corollary}\label{cor:HPD}
Let $X_L$ and $Y_L$ be as in Theorem \ref{thm:HPD}. If $\mathrm{dim}(Y_L)\leq 1$, then the (mixed) motive $M(X_L)_\bbQ$ is Schur-finite.
\end{corollary}
In the next subsections we illustrate the strength of Corollary \ref{cor:HPD} in several examples (in all the cases below $k$ is algebraically closed and of characteristic zero).
\subsection*{Grassmannian $\mathrm{Gr}(2,5)$}
Let $X$ be the Grassmannian $\mathrm{Gr}(2,5)$ equipped with the Pl\"ucker embedding $\mathrm{Gr}(2,5) \to \bbP(\wedge^2W)$, where $W$ is a $5$-dimensional $k$-vector space. As explained in \cite[\S6.1]{Hyperplane}, the category $\perf(\mathrm{Gr}(2,5))$ admits a rectangular Lefschetz decomposition $\langle \bbA_0, \bbA_1(1),\ldots, \bbA_4(4)\rangle$ and $\bbA_0$ a full exceptional collection of length $2$. Moreover, the HP-dual $Y$ of $\mathrm{Gr}(2,5)$ is the dual Grassmannian $\mathrm{Gr}(2,W^\ast)$. Given a generic linear subspace $L\subset \wedge^2W^\ast$, consider the associated smooth linear sections $\mathrm{Gr}(2,5)_L$ and $\mathrm{Gr}(2,W^\ast)_L$. Making use of Corollary \ref{cor:HPD} and of the equalities $\mathrm{dim}(\mathrm{Gr}(2,5))=6$, $\mathrm{dim}(\mathrm{Gr}(2,5)_L)=6-\mathrm{dim}(L)$, $\mathrm{dim}(\mathrm{Gr}(2,W^\ast)_L)=\mathrm{dim}(L)-4$, we hence obtain the following result:
\begin{theorem}\label{thm:app1}
The (mixed) motive $M(\mathrm{Gr}(2,5)_L)_\bbQ$ of a smooth linear section of $\mathrm{Gr}(2,5)$ of arbitrary codimension is Schur-finite.
\end{theorem}
\subsection*{Grassmannians $\mathrm{Gr}(2,6)$ and $\mathrm{Gr}(2,7)$}
Let $W$ be a $k$-vector space of dimension $6$, resp. $7$, and $X$ the Grassmannian $\mathrm{Gr}(2,6)$, resp. $\mathrm{Gr}(2,7)$, equipped with the Pl\"ucker embedding $\mathrm{Gr}(2,6) \to \bbP(\wedge^2W)$, resp. $\mathrm{Gr}(2,7) \to \bbP(\wedge^2W)$. As explained in \cite[\S10]{Lines}, the category $\perf(\mathrm{Gr}(2,6))$ admits a Lefschetz decomposition $\langle \bbA_0, \bbA_1(1),\ldots, \bbA_5(5)\rangle$, with $\bbA_0= \bbA_1= \bbA_2$ and $\bbA_3=\bbA_4 = \bbA_5$. Moreover, $\bbA_0$ and $\bbA_3$ admit full exceptional collections of length $3$ and $2$, respectively. In the same vein, as explained in \cite[\S11]{Lines}, the $\perf(\mathrm{Gr}(2,7))$ admits~a rectangular Lefschetz decomposition $\langle \bbA_0, \bbA_1(1), \ldots, \bbA_6(6)\rangle$ and $\bbA_0$ a full exceptional collection of length $3$. Furthermore, the HP-dual $Y$ of $\mathrm{Gr}(2,6)$, resp. $\mathrm{Gr}(2,7)$, is given by $\perf(\mathrm{Pf}(4,W^\ast);\cF)$, where $\mathrm{Pf}(4,W^\ast) \subset \bbP(\wedge^2W^\ast)$ is the (singular) Pfaffian variety and $\cF$ a certain coherent sheaf of algebras on $\mathrm{Pf}(4,W^\ast)$. The singular locus of $\mathrm{Pf}(4,W^\ast)$ is $8$-dimensional, resp. $10$-dimensional, and $\cF$ is Morita equivalent to the structure sheaf on the smooth locus. Therefore, given a generic linear subspace $L \subset \wedge^2 W^\ast$ of dimension $\leq 6$, resp. $\leq 10$, we can consider the associated smooth linear sections $\mathrm{Gr}(2,6)_L$ and $\mathrm{Pf}(4,W^\ast)_L$, resp. $\mathrm{Gr}(2,7)_L$ and~$\mathrm{Pf}(4,W^\ast)_L$. Making use of Corollary \ref{cor:HPD} and of the equalities $\mathrm{dim}(\mathrm{Gr}(2,6))=8$, $\mathrm{dim}(\mathrm{Gr}(2,6)_L)=8-\mathrm{dim}(L)$, $\mathrm{dim}(\mathrm{Pf}(4,W^\ast)_L)=\mathrm{dim}(L)-2$, resp. $\mathrm{dim}(\mathrm{Gr}(2,7))=10$, $\mathrm{dim}(\mathrm{Gr}(2,7)_L)=~\!10-\mathrm{dim}(L)$, $\mathrm{dim}(\mathrm{Pf}(4,W^\ast)_L)=\mathrm{dim}(L)-4$, we hence obtain the following result:
\begin{theorem}\label{thm:app2}
\begin{itemize}
\item[(i)] The (mixed) motive $M(\mathrm{Gr}(2,6)_L)_\bbQ$ of a smooth linear section of $\mathrm{Gr}(2,6)$ of codimension $1$, $2$, or $3$, is Schur-finite.
\item[(ii)] The (mixed) motive $M(\mathrm{Gr}(2,7)_L)_\bbQ$ of a smooth linear section of $\mathrm{Gr}(2,7)$ of codimension $1$, $2$, $3$, $4$, or $5$, is Schur-finite.
\end{itemize}
\end{theorem}
\subsection*{Lagrangian Grassmannian $\mathrm{LGr}(3,6)$}
Let $W$ be a $6$-dimensional $k$-vector space, equipped with a symplectic form $\omega$, and $X$ the associated Lagrangian Grassmannian $\mathrm{LGr}(3,6):=\mathrm{LGr}(3,W)$. The Pl\"ucker embedding $\mathrm{Gr}(3,W) \subset \bbP(\wedge^3 W)$ restricts to an embedding $\mathrm{LGr}(3,6) \to \bbP(V)$ into a $13$-dimensional projective space; see \cite[\S6.3]{Hyperplane}. The classical projective dual variety $\mathrm{LGr}(3,6)^\vee \subset \bbP(V^\ast)$ is a quartic hypersurface which is singular along a $9$-dimensional subvariety $Z$. As explained in {\em loc. cit.}, the category $\perf(\mathrm{LGr}(3,6))$ admits a rectangular Lefschetz decomposition $\langle \bbA_0, \bbA_1(1), \cdots, \bbA_3(3)\rangle$ and $\bbA_0$ a full exceptional collection of length $2$.  Moreover, the HP-dual $Y$ of $\mathrm{LGr}(3,6)$ is given by $\perf(\mathrm{LGr}(3,6)^\vee\backslash Z; \cF)$, where $\cF$ is a certain sheaf of Azumaya algebras on $\mathrm{LGr}(3,6)^\vee\backslash Z$. Given a generic linear subspace $L \subset V^\ast$ such that $\bbP(L)\cap Z=\emptyset$, consider the associated smooth linear sections $\mathrm{LGr}(3,6)_L$ and $(\mathrm{LGr}(3,6)^\vee\backslash Z)_L$. Making use of Corollary \ref{cor:HPD}, of the equalities $\mathrm{dim}(\mathrm{LGr}(3,6))=6$, $\mathrm{dim}(\mathrm{LGr}(3,6)_L)=6-\mathrm{dim}(L)$, $\mathrm{dim}((\mathrm{LGr}(3,6)^\vee \backslash Z)_L)=\mathrm{dim}(L)-2$, and of the fact that the Brauer group of a smooth curve is trivial, we hence obtain the following result:
\begin{theorem}\label{thm:app4}
The (mixed) motive $M(\mathrm{LGr}(3,6)_L)_\bbQ$ of a smooth linear section of $\mathrm{LGr}(3,6)$ of codimension $1$, $2$, or $3$, is Schur-finite.
\end{theorem}
\subsection*{Spinor variety $\mathrm{Sp_+}(5,10)$}
Let $W$ be a $10$-dimensional $k$-vector space and $q \in S^2W^\ast$ a nondegenerate quadratic form. The associated isotropic Grassmannian of $5$-dimensional subspaces in $W$ has two (isomorphic) connected components $X:=\mathrm{Sp}_+(5,10)\subset \bbP(\wedge^5 W)$ and $Y:=\mathrm{Sp}_-(5,10)\subset \bbP(\wedge^5 W^\ast)$ called the {\em Spinor varieties}. As explained in \cite[\S6.2]{Hyperplane}, the category $\perf(\mathrm{Sp}_+(5,10))$ admits a rectangular Lefschetz decomposition $\langle \bbA_0, \bbA_1(1), \ldots, \bbA_7(7)\rangle$ and $\bbA_0$ a full exceptional collection of length $2$. Moreover, the spinor varieties $\mathrm{Sp}_+(5,10)$ and $\mathrm{Sp}_-(5,10)$ are HP-dual to each other. Given a generic linear subspace $L\subset \wedge^5 W^\ast$, consider the associated smooth linear sections $\mathrm{Sp}_+(5,10)_L$ and $\mathrm{Sp}_-(5,10)_L$. Making use of Corollary \ref{cor:HPD} and of the equalities $\mathrm{dim}(\mathrm{Sp}_+(5,10))=10$, $\mathrm{dim}(\mathrm{Sp}_+(5,10)_L)=10-\mathrm{dim}(L)$, $\mathrm{dim}(\mathrm{Sp}_-(5,10)_L)=\mathrm{dim}(L)-6$, we hence obtain the following result:
\begin{theorem}\label{thm:app5}
The (mixed) motive $M(\mathrm{Sp}_+(5,10)_L)_\bbQ$ of a smooth linear section of $\mathrm{Sp}_+(5,10)$ of codimension $1$, $2$, $3$, $4$, $5$, $6$, or $7$, is Schur-finite.
\end{theorem}
\subsection*{Determinantal varieties}
Let $U$ and $V$ be two $k$-vector spaces of dimensions $m$ and $n$, respectively, with $m\leq n$, $W$ the tensor product $U \otimes V$, and $0<r <m$ an integer. Consider the determinantal variety $\mathcal{Z}^r_{m,n} \subset \bbP(W)$, resp. $\cW^r_{m,n}\subset \bbP(W^\ast)$, defined as the locus of those matrices $V \to U^\ast$, resp. $V^\ast \to U$, with rank at most $r$, resp. with corank at least $r$. For example, $\cZ^1_{m,n}$ are the classical Segre varieties. As explained by Bernardara, Bolognesi, and Faenzi in \cite[\S3]{BBF}, $\mathcal{Z}^r_{m,n}$ and $\cW^r_{m,n}$ admit (Springer) resolutions of singularities $X:=\cX^r_{m,n}$ and $Y:=\cY^r_{m,n}$, respectively. Moreover, the category $\perf(\cX^r_{m,n})$ admits a rectangular Lefschetz decomposition $\langle \bbA_0, \bbA_1(1), \ldots, \bbA_{nr-1}\rangle$ and $\bbA_0$ a full exceptional collection of length $\binom{m}{r}$. Furthermore, the resolutions $\cX^r_{m,n}$ and $\cY^r_{m,n}$ are HP-dual to each other. Given a generic linear subspace $L\subset W^\ast$, consider the associated smooth linear sections $(\cX^r_{m,n})_L$ and $(\cY^r_{m,n})_L$. Making use of Corollary \ref{cor:HPD} and of the equalities $\mathrm{dim}(\cX^r_{m,n})=r(n+m-r)-1$, $\mathrm{dim}((\cX^r_{m,n})_L)=r(n+m-r)-1-\mathrm{dim}(L)$, $\mathrm{dim}((\cY^r_{m,n})_L)=r(m-n-r)-1+\mathrm{dim}(L)$, we hence obtain the following result:
\begin{theorem}\label{thm:app6}
If $r(m-n-r)-1+\mathrm{dim}(L)\leq 1$, then the (mixed) motive $M((\cX^r_{m,n})_L)_\bbQ$ of a linear section of $\cX^r_{m,n}$ of codimension $\mathrm{dim}(L)$ is Schur-finite.
\end{theorem}
Since $0 < r<m\leq n$, the inequality of Theorem \ref{thm:app6} holds whenever the dimension of $L$ is equal to $1$, $2$, or $3$.  This leads to the following unconditional result:
\begin{corollary}\label{cor:unconditional}
The (mixed) motive $M((\cX^r_{m,n})_L)_\bbQ$ of a smooth linear section of $\cX^r_{m,n}$ of codimension $1$, $2$, or $3$, is Schur-finite.
\end{corollary}
The dimension of $\cX^r_{m,n}$, \ie the integer $r(n+m-r)-1$, can be arbitrary high. Consequently, Corollary \ref{cor:unconditional} furnish us infinitely many examples of smooth projective $k$-schemes, of arbitrary high dimension, whose (mixed) motives are Schur-finite.
\begin{remark}\label{rk:Folds}
\begin{itemize}
\item[(i)] To the best of the author's knowledge, the above Theorems \ref{thm:app1}-\ref{thm:app6} (and Corollary \ref{cor:unconditional}) are new in the literature. They provide us several new examples of Schur-finite (mixed) motives.
\item[(ii)]  As proved by Gorchinskiy and Guletskii in \cite[\S5]{Threefolds}, the (mixed) motives of Fano threefolds are Schur-finite. In the particular case of codimension $3$ at Theorems \ref{thm:app1} and \ref{thm:app4}, and of codimension $7$ at Theorem \ref{thm:app5}, the corresponding smooth linear sections $X_L$ are Fano threefolds. We hence obtain, in these particular cases, an alternative proof of Schur-finiteness.
\end{itemize}
\end{remark}
\subsection*{Kimura-finiteness}
Let $(\cC,\otimes, {\bf 1})$ be a $\bbQ$-linear, idempotent complete, symmetric monoidal category. In the case of the partition $\lambda=(1,\ldots, 1)$, resp. $\lambda=(n)$, the associated Schur-functor $\wedge^n:=S_{(1,\ldots, 1)}$, resp. $\mathrm{Sym}^n:=S_{(n)}$, is called the {\em $n^{\mathrm{th}}$ wedge product}, resp. the {\em $n^{\mathrm{th}}$ symmetric product}. Following Kimura \cite{Kimura}, an object $a \in \cC$ is called {\em even-dimensional}, resp. {\em odd-dimensional}, if $\wedge^n(a)$, resp. $\mathrm{Sym}^n(a)=0$, for some $n \gg 0$. The biggest integer $\mathrm{kim}_+(a)$, resp. $\mathrm{kim}_-(a)$, for which $\wedge^{\mathrm{kim}_+(a)}\neq 0$, resp. $\mathrm{Sym}^{\mathrm{kim}_-(a)}(a)\neq 0$, is called the {\em even}, resp. {\em odd}, {\em Kimura-dimension of $a$}. An object $a \in \cC$ is called {\em Kimura-finite} if $a\simeq a_+\oplus a_-$, with $a_+$ even-dimensional and $a_-$ odd-dimensional. The integer $\mathrm{kim}(a)=\mathrm{kim}_+(a_+)+\mathrm{kim}_-(a_-)$ is called the {\em Kimura-dimension of $a$}. Finally, Kimura-finiteness implies Schur-finiteness.

The notion of Kimura-finiteness has been extensively studied in the motivic setting; consult the survey \cite{Andre-survey}. For example, Kimura proved in \cite[\S4]{Kimura} that the (mixed) motives $M(Z)_\bbQ$ of smooth projective $k$-schemes $Z$ of dimension $\leq 1$ are Kimura-finite. Moreover, we have the following computations
$$
\mathrm{kim}_+(M(Z)_{\bbQ,+})= \begin{cases}
\mathrm{length}(Z) & \mathrm{if}\,\,\mathrm{dim}(Z)=0\\
2&\mathrm{if}\,\,\mathrm{dim}(Z)=1
\end{cases}
$$
$$
\mathrm{kim}_-(M(Z)_{\bbQ,-})= \begin{cases}
0 & \mathrm{if}\,\,\mathrm{dim}(Z)=0\\
2g&\mathrm{if}\,\,\mathrm{dim}(Z)=1\,,
\end{cases}
$$
where $g$ stands for the genus of the smooth projective curve $Z$; when $Z=\emptyset$, we have $\mathrm{kim}(M(Z)_\bbQ)=0$. As another example, Guletskii and Pedrini proved in \cite[\S4]{GP} that the (mixed) motive $M(Z)_\bbQ$ of a smooth projective surface $Z$, with $p_g(Z)=0$, is Kimura-finite if and only if Bloch's conjecture on the Albanese kernel for $Z$ holds. It also should be pointed out that, in contrast with Schur-finiteness, it is known that not every mixed motive\footnote{Nevertheless, it is conjectured that every Chow motive is Kimura-finite; see \cite[Conj.~2.7]{Andre}.} is Kimura-finite!; consult \cite[\S5.1]{Mazza} for a counter-example.
\begin{theorem}\label{thm:2}
Theorems \ref{thm:app1}-\ref{thm:app6} (and Corollary \ref{cor:unconditional}) hold {\em mutatis mutandis} with Schur-finiteness replaced by Kimura-finitess. Moreover, we have the equalities
\begin{eqnarray}
\mathrm{kim}_+(M(X_L)_{\bbQ,+}) & = & \mathrm{kim}_+(M(Y_L)_{\bbQ,+}) + (l_{\mathrm{dim}(L)}+ \cdots + l_{i-1}) \label{eq:equality}\\
\mathrm{kim}_-(M(X_L)_{\bbQ,-}) & = & \mathrm{kim}_-(M(Y_L)_{\bbQ,-}) \label{eq:equality-new} \,,
\end{eqnarray}
where $l_r$ stands for the length of the full exceptional collection of the category $\bbA_r$. 
\end{theorem}
Note that the sum $l_{\mathrm{dim}(L)}+ \cdots + l_{i-1}$ reduces to $2(i-\mathrm{dim}(L))$ in Theorems \ref{thm:app1} and \ref{thm:app4}-\ref{thm:app5}, to $3(i-\mathrm{dim}(L))$ in Theorem \ref{thm:app2}(ii), and to $\binom{m}{r}(i-\mathrm{dim}(L))$ in Theorem \ref{thm:app6}. To the best of the author's knowledge, Theorem \ref{thm:2} is new in the literature. It not only provides us several new examples of Kimura-finite (mixed) motives but also computes the corresponding even/odd Kimura-dimensions.
\begin{example}
\begin{itemize}
\item[(i)] In the case of codimension $3$ at Theorem \ref{thm:app2}(i), $\mathrm{Gr}(2,6)_L$ is a fivefold and $\mathrm{Pf}(4,W^\ast)_L$ an elliptic curve; see \cite[page~33]{Lines}. Consequently, $\mathrm{kim}_+(M(\mathrm{Gr}(2,6)_L)_{\bbQ,+})=8$ and $\mathrm{kim}_-(M(\mathrm{Gr}(2,6)_L)_{\bbQ,-})=2$.
\item[(ii)] In the case of codimension $5$ at Theorem \ref{thm:app2}(ii), $\mathrm{Gr}(2,7)_L$ is a fivefold and $\mathrm{Pf}(4,W^\ast)_L$ a smooth projective curve of genus $43$; see \cite[page~35]{Lines}. Consequently, $\mathrm{kim}_+(M(\mathrm{Gr}(2,7)_L)_{\bbQ,+})= 8$ and $\mathrm{kim}_-(M(\mathrm{Gr}(2,7)_L)_{\bbQ,-})= 86$.
\end{itemize}
\end{example}
Let $K_0(\mathrm{DM}_{\mathrm{gm}}(k)_\bbQ)$ be the Grothendieck ring of the symmetric monoidal triangulated category of mixed motives. Following Kapranov \cite{Kapranov}, given any mixed motive $M \in \mathrm{DM}_{\mathrm{gm}}(k)_\bbQ$, we can consider the associated motivic zeta function $\zeta(M;t):=\sum_{n=0}^\infty [\mathrm{Sym}^n(M)]t^n$. Since the motivic zeta function of every Kimura-finite mixed motive is rational (see \cite[Prop.~4.6]{Andre-survey}), we obtain the following result:
\begin{corollary}
Let $X_L$ be as in Theorems \ref{thm:app1}-\ref{thm:app6}. Then, the associated motivic zeta function is rational $\zeta(M(X_L)_\bbQ; t) = \frac{1+p(t)t}{1+q(t)t}$, with $p(t)$, resp. $q(t)$, a polynomial of degree $\mathrm{kim}_-(M(X_L)_{\bbQ, -})-1$, resp. $\mathrm{kim}_+(M(X_L)_{\bbQ, +})-1$.
\end{corollary}
Finally, it should be mentioned that the above Remark \ref{rk:Folds}(ii) also holds {\em mutatis mutandis} with Schur-finiteness replaced by Kimura-finitess.

\section{Preliminaries}
\subsection{Dg categories}\label{sub:dg}
For a survey on dg categories consult Keller's ICM talk \cite{ICM-Keller}. Let $\cC(k)$ be the category of complexes of $k$-vector spaces. A {\em dg category $\cA$} is a category enriched over $\cC(k)$ and a {\em dg functor $F\colon \cA \to \cB$} is a functor enriched over $\cC(k)$. Every (dg) $k$-algebra $A$ gives naturally rise to a dg category with a single object. Another source of examples is provided by schemes since the category of perfect complexes $\perf(Z)$ of every quasi-compact quasi-separated $k$-scheme $Z$ admits a canonical dg enhancement\footnote{When $X$ is quasi-projective this dg enhancement is unique; see Lunts-Orlov \cite[Thm.~2.12]{LO}.} $\perf_\dg(Z)$. Following Kontsevich \cite{Miami,finMot,IAS}, a dg category $\cA$ is called {\em smooth} if it is perfect as a bimodule over itself and {\em proper} if $\sum_j \mathrm{dim}\, H^j\cA(x,y)< \infty$ for any pair of objects $(x,y)$. Examples include the dg categories of perfect complexes $\perf_\dg(Z)$ associated to smooth proper $k$-schemes $Z$. Let $\dgcat_{\mathrm{sp}}(k)$ be the category of smooth proper dg categories and dg functors.

\subsection{Noncommutative Chow motives}\label{sub:homological}
For a book on noncommutative motives consult \cite{book}. Recall from \cite[\S4.1]{book} the construction of the additive category of noncommutative Chow motives $\NChow(k)_\bbQ$. By construction, this category is $\bbQ$-linear, idempotent complete, symmetric monoidal, and comes equipped with a symmetric monoidal functor $U(-)_\bbQ\colon \dgcat_{\mathrm{sp}}(k) \to \NChow(k)_\bbQ$.

\subsection{Orbit categories}\label{sub:orbit}
Let $(\cC, \otimes, {\bf 1})$ be a $\bbQ$-linear, additive, symmetric monoidal category, and $\cO \in \cC$ a $\otimes$-invertible object. The {\em orbit category} $\cC/_{\!-\otimes \cO}$ has the same objects as $\cC$ and morphisms $\Hom_{\cC/_{\!-\otimes \cO}}(a,b):=\oplus_{n \in \bbZ} \Hom_\cC(a, b \otimes \cO^{\otimes n})$. Given objects $a, b, c$ and morphisms $\mathrm{f}=\{f_n\}_{n \in \bbZ}$ and $\mathrm{g}=\{g_n\}_{n \in \bbZ}$, the $j^{\mathrm{th}}$-component of $\mathrm{g}\circ \mathrm{f}$ is defined as $\sum_n (g_{j -n} \otimes \cO^{\otimes n})\circ f_n$. By construction, we have the canonical functor $\pi\colon \cC \to \cC/_{\!-\otimes \cO}$, given by $a \mapsto a$ and $f \mapsto \mathrm{f}=\{f_n\}_{n \in \bbZ}$, where $f_0=f$ and $f_n=0$ if $n\neq 0$. Moreover, the category $\cC/_{\!-\otimes \cO}$ is $\bbQ$-linear, additive, and inherits from $\cC$ a symmetric monoidal structure making the functor $\pi$ symmetric monoidal.

\section{Proof of Theorem \ref{thm:HPD}}
By definition of the Lefschetz decomposition $\langle \bbA_0, \bbA_1(1), \ldots, \bbA_{i-1}(i-1)\rangle$, we have a chain of admissible triangulated subcategories $\bbA_{i-1}\subseteq \cdots \subseteq \bbA_1\subseteq \bbA_0$ with $\bbA_r(r):=\bbA_r \otimes \cO_X(r)$. Note that the category $\bbA_r(r)$ is equivalent to $\bbA_r$. Let $\mathfrak{a}_r$ be the right orthogonal complement to $\bbA_{r+1}$ in $\bbA_r$; these are called the {\em primitive subcategories} in \cite[\S4]{KuznetsovHPD}. By definition, we have semi-orthogonal decompositions:
\begin{eqnarray}\label{eq:decomp1}
\bbA_r = \langle \mathfrak{a}_r, \mathfrak{a}_{r+1}, \ldots, \mathfrak{a}_{i-1} \rangle && 0\leq r \leq i-1\,.
\end{eqnarray}
As proved in \cite[Thm.~6.3]{KuznetsovHPD}, the category $\perf(Y)$ admits a HP-dual Lefschetz decomposition $\langle \bbB_{j-1}(1-j), \bbB_{j-2}(2-j), \ldots, \bbB_0\rangle$ with respect to $\cO_Y(1)$. As above, we have a chain of admissible subcategories $\bbB_{j-1} \subseteq \bbB_{j-2} \subseteq \cdots \subseteq \bbB_0$. Moreover, the primitive subcategories coincide (via a Fourier-Mukai functor) with those of $\perf(X)$ and we have semi-orthogonal decompositions:
\begin{eqnarray}\label{eq:decomp2}
\bbB_r=\langle \mathfrak{a}_0, \mathfrak{a}_1, \ldots, \mathfrak{a}_{\mathrm{dim}(V)-r-2}\rangle && 0 \leq r \leq j-1\,.
\end{eqnarray}
Furthermore, the assumptions $\mathrm{dim}(X_L)=\mathrm{dim}(X)-\mathrm{dim}(L)$ and $\mathrm{dim}(Y_L)=\mathrm{dim}(Y) - \mathrm{dim}(L^\perp)$ imply the existence of semi-orthogonal decompositions
\begin{equation}\label{eq:semi-1}
\perf(X_L)=\langle \bbC_L, \bbA_{\mathrm{dim}(L)}(1), \ldots, \bbA_{i-1}(i-\mathrm{dim}(L))\rangle 
\end{equation}
\begin{equation}\label{eq:semi-2}
\perf(Y_L)=\langle \bbB_{j-1}(\mathrm{dim}(L^\perp)-j), \ldots, \bbB_{\mathrm{dim}(L^\perp)}(-1), \bbC_L \rangle\,,
\end{equation}
where $\bbC_L$ is a common triangulated category. Let us denote by $\bbC_L^\dg$, $\bbA_r^\dg$, and $\mathfrak{a}_r^\dg$, the dg enhancement of $\bbC_L$, $\bbA_r$, and $\mathfrak{a}_r$, induced from $\perf_\dg(X_L)$. Similarly, let us denote by $\bbC_L^{\dg'}$ and $\bbB_r^{\dg}$ the dg enhancement of $\bbC_L$ and $\bbB_r$ induced from $\perf_\dg(Y_L)$. Note that since $X_L$ and $Y_L$ are smooth projective $k$-schemes, all the preceding dg categories are smooth and proper. As explained in \cite[\S2.4.1]{book}, the above semi-orthogonal decompositions \eqref{eq:semi-1}-\eqref{eq:semi-2} give rise to the following direct sum decompositions of noncommutative Chow motives:
$$ U(\perf_\dg(X_L))_\bbQ \simeq U(\bbC^\dg_L)_\bbQ \oplus U(\bbA^\dg_{\mathrm{dim}(L)})_\bbQ \oplus \cdots \oplus U(\bbA^\dg_{i-1})_\bbQ$$
$$ U(\perf_\dg(Y_L))_\bbQ \simeq U(\bbB_{j-1}^\dg)_\bbQ \oplus \cdots \oplus U(\bbB^\dg_{\mathrm{dim}(L^\perp)})_\bbQ \oplus U(\bbC_L^{\dg'})_\bbQ\,.$$
Since by assumption the triangulated category $\bbA_0$ admits a full exceptional collection, the noncommutative Chow motive $U(\bbA_0^\dg)_\bbQ$ is isomorphic to a finite direct sum of copies of $U(k)_\bbQ$; see \cite[\S2.4.2]{book}. In particular, it is Schur-finite. Making use of the semi-orthogonal decompositions \eqref{eq:decomp1}-\eqref{eq:decomp2}, we hence conclude that the noncommutative Chow motives $U(\bbA^\dg_{\mathrm{dim}(L)})_\bbQ, \ldots, U(\bbA^\dg_{i-1})_\bbQ$ and $U(\bbB^\dg_{j-1})_\bbQ, \ldots, U(\bbB^\dg_{\mathrm{dim}(L^\perp)})_\bbQ$ are also Schur-finite. This implies that $U(\perf_\dg(X_L))_\bbQ$ is Schur-finite if and only if $U(\bbC_L^\dg)_\bbQ$ is Schur-finite and, similarly, that $U(\perf_\dg(Y_L))_\bbQ$ is Schur-finite if and only if $U(\bbC_L^{\dg'})_\bbQ$ is Schur-finite. Since the functor $\perf(X_L) \to \bbC_L \to \perf(Y_L)$ is of Fourier-Mukai type, the dg categories $\bbC_L^\dg$ and $\bbC_L^{\dg'}$ are Morita equivalent. Using the fact that the functor $U(-)_\bbQ$ inverts Morita equivalences (see \cite[\S1.6 and Thm.~2.9]{book}), we hence conclude that $U(\bbC^\dg_L)_\bbQ \simeq U(\bbC^{\dg'}_L)_\bbQ$. Consequently, the proof of Theorem \ref{thm:HPD} follows now automatically from the following result:
\begin{proposition}\label{prop:aux}
Given a smooth projective $k$-scheme projective $k$-scheme $Z$, the (mixed) motive $M(Z)_\bbQ$ is Schur-finite if and only if the noncommutative Chow motive $U(\perf_\dg(Z))_\bbQ$ is Schur-finite.
\end{proposition}
\begin{proof}
Recall from \cite[\S4.1]{Andre} the construction of the classical category of Chow motives $\Chow(k)_\bbQ$. This category is $\bbQ$-linear, idempotent complete, symmetric monoidal, and comes equipped with a (contravariant) symmetric monoidal functor $\mathfrak{h}(-)_\bbQ\colon \mathrm{SmProj}(k)^\op \to \Chow(k)_\bbQ$, defined on smooth projective $k$-schemes. As proved in \cite[Thm.~1.1]{CvsNC} (see also \cite[Thm.~4.3]{book}), there exists a $\bbQ$-linear, fully-faithful, symmetric monoidal functor $\Phi$ making the following diagram commute
\begin{equation}\label{eq:diagram1}
\xymatrix{
\mathrm{SmProj}(k)^\op \ar[rr]^-{Z\mapsto \perf_\dg(Z)} \ar[d]_-{\mathfrak{h}(-)_\bbQ} && \dgcat_{\mathrm{sp}}(k) \ar[dd]^-{U(-)_\bbQ} \\
\Chow(k)_\bbQ \ar[d]_-\pi && \\
\Chow(k)_\bbQ/_{\!-\otimes \bbQ(1)} \ar[rr]_-{\Phi} && \NChow(k)_\bbQ\,,
}
\end{equation}
where $\bbQ(1)$ stands for the Tate motive. Since the functor $\pi$ is faithful and the functor $\Phi$ is fully-faithful, it follows from Lemma \ref{lem:aux} below and from the commutative diagram \eqref{eq:diagram1} that the Chow motive $\mathfrak{h}(Z)_\bbQ$ is Schur-finite if and only the noncommutative Chow motive $U(\perf_\dg(Z))_\bbQ$ is Schur-finite.

The category of Chow motives $\Chow(k)_\bbQ$ is not only symmetric monoidal but moreover {\em rigid}, \ie all its objects are (strongly) dualizable. Let us denote by $(-)^\vee\colon \Chow(k)^\op_\bbQ \stackrel{\simeq}{\too} \Chow(k)_\bbQ$ the (contravariant) duality auto-equivalence and by $\mathfrak{h}(-)_\bbQ^\vee$ the (covariant) composition $(-)^\vee \circ \mathfrak{h}(-)_\bbQ$. As proved by Voevodsky in \cite[Prop.~2.1.4 and Cor.~4.2.6]{Voevodsky} (see also \cite[Thm.~18.3.1.1]{Andre}), there exists a $\bbQ$-linear, fully-faithful, symmetric monoidal functor $\Psi$ making the diagram commute:
\begin{equation}\label{eq:diagram2}
\xymatrix{
\mathrm{SmProj}(k) \ar[rr]^-{Z \mapsto Z} \ar[d]_-{\mathfrak{h}(-)_\bbQ^\vee} && \mathrm{Sm}(k) \ar[d]^-{M(-)_\bbQ} \\
\Chow(k)_\bbQ \ar[rr]_-{\Psi} && \mathrm{DM}_{\mathrm{gm}}(k)_\bbQ\,.
}
\end{equation}
Since Schur-finiteness is stable under duality and $\Psi$ is fully-faithful, it follows then from Lemma \ref{lem:aux} below and from the commutative diagram \eqref{eq:diagram2} that the Chow motive $\mathfrak{h}(Z)_\bbQ$ is Schur-finite if and only if the (mixed) motive $M(Z)_\bbQ$ is Schur-finite. This concludes the proof of Proposition \ref{prop:aux}.
\end{proof}
\begin{lemma}\label{lem:aux}
Let $(\cC, \otimes, {\bf 1})$ and $(\cC',\otimes', {\bf 1}')$ be two $\bbQ$-linear, idempotent complete, symmetric monoidal categories, and $H\colon \cC \to \cC'$ a $\bbQ$-linear, symmetric monoidal functor. Given any object $a \in \cC$, the following holds:
\begin{itemize}
\item[(i)] If $a$ is Schur-finite, then $H(a)$ is also Schur-finite.
\item[(ii)] If $H$ is faithful and $H(a)$ is Schur-finite, then $a$ is also Schur-finite.
\end{itemize}
\end{lemma}
\begin{proof}
The proof is a simple exercise which we leave to the reader.
\end{proof}

\section{Proof of Theorem \ref{thm:2}}
In the particular cases of Theorems \ref{thm:app1}-\ref{thm:app6}, we have $\mathrm{dim}(Y_L)\leq 1$ (in some cases $Y_L=\emptyset$) and the semi-orthogonal decomposition \eqref{eq:semi-1} reduces to
\begin{equation}\label{eq:decomposition}
\perf(X_L)=\langle \perf(Y_L), \bbA_{\mathrm{dim}(L)}, \ldots, \bbA_{i-1}(i-\mathrm{dim}(L))\rangle\,,
\end{equation}
\ie the common triangulated category $\bbC_L$ agrees with $\perf(Y_L)$. Recall that the triangulated category $\bbA_r$ admits a full exceptional collection of length $l_r$. Similarly to the proof of Theorem \ref{thm:HPD}, the semi-orthogonal decomposition \eqref{eq:decomposition} gives then rise to the following direct sum decomposition of noncommutative Chow motives
\begin{equation}\label{eq:iso}
U(\perf_\dg(X_L))_\bbQ \simeq U(\perf_\dg(Y_L))_\bbQ \oplus \oplus_{r=\mathrm{dim}(L)}^{i-1} U(k)^{\oplus l_r}_\bbQ\,.
\end{equation} 
Thanks to the commutative diagram \eqref{eq:diagram1} and to the fact that the functor $\Phi$ is fully-faithful, \eqref{eq:iso} yields an isomorphism in the orbit category of Chow motives
$$ \pi(\mathfrak{h}(X_L)_\bbQ)\simeq \pi(\mathfrak{h}(Y_L)_\bbQ \oplus \oplus_{r=\mathrm{dim}(L)}^{i-1} \mathfrak{h}(\mathrm{Spec}(k))_\bbQ^{\oplus l_r})\,.$$
Therefore, by definition of the orbit category, there exist morphisms
$$ \mathrm{f}=\{f_n\}_{n \in \bbZ} \in \Hom_{\Chow(k)_\bbQ}(\mathfrak{h}(X_L)_\bbQ, (\mathfrak{h}(Y_L)_\bbQ \oplus \oplus_{r=\mathrm{dim}(L)}^{i-1} \mathfrak{h}(\mathrm{Spec}(k))^{\oplus l_r}_\bbQ)(n))$$
$$ \mathrm{g}=\{g_n\}_{n \in \bbZ} \in \Hom_{\Chow(k)_\bbQ}(\mathfrak{h}(Y_L)_\bbQ \oplus \oplus_{r=\mathrm{dim}(L)}^{i-1} \mathfrak{h}(\mathrm{Spec}(k))^{\oplus l_r}_\bbQ, \mathfrak{h}(X_L)_\bbQ(n))$$
verifying the equalities $\mathrm{g}\circ \mathrm{f}=\id=\mathrm{f}\circ \mathrm{g}$; in order to simplify the exposition, we (will) write $-(n)$ instead of $- \otimes \bbQ(1)^{\otimes n}$. Recall that by definition of the category of Chow motives, we have $f_n=0$ if $n \notin \{-\mathrm{dim}(X_L), \ldots,  \mathrm{dim}(Y_L)\}$ and $g_n=0$ if $n \in \{-\mathrm{dim}(Y_L), \ldots, \mathrm{dim}(X_L)\}$. The sets $\{f_n\,|\, - \mathrm{dim}(X_L) \leq n \leq \mathrm{dim}(Y_L)\}$ and $\{g_{-n}(n)\,|\,-\mathrm{dim}(X_L)\leq n \leq \mathrm{dim}(Y_L)\}$ give then rise to the following morphisms
$$ \alpha\colon \mathfrak{h}(X_L)_\bbQ \too \oplus_{n=-\mathrm{dim}(X_L)}^{\mathrm{dim}(Y_L)} (\mathfrak{h}(Y_L)_\bbQ \oplus \oplus_{r=\mathrm{dim}(L)}^{i-1} \mathfrak{h}(\mathrm{Spec}(k))^{\oplus l_r}_\bbQ)(n)$$
$$\beta\colon \oplus_{n=-\mathrm{dim}(X_L)}^{\mathrm{dim}(Y_L)} (\mathfrak{h}(Y_L)_\bbQ \oplus \oplus_{r=\mathrm{dim}(L)}^{i-1} \mathfrak{h}(\mathrm{Spec}(k))^{\oplus l_r}_\bbQ)(n) \too \mathfrak{h}(X_L)_\bbQ\,.$$
The composition $\beta\circ \alpha$ agrees with the $0^{\mathrm{th}}$ component of $\mathrm{g}\circ \mathrm{f}=\id$, \ie with the identity of $\mathfrak{h}(X_L)_\bbQ$. Consequently, $\mathfrak{h}(X_L)_\bbQ$ is a direct summand of the direct sum
\begin{equation}\label{eq:directsum}
\oplus_{n=-\mathrm{dim}(X_L)}^{\mathrm{dim}(Y_L)} (\mathfrak{h}(Y_L)_\bbQ \oplus \oplus_{r=\mathrm{dim}(L)}^{i-1} \mathfrak{h}(\mathrm{Spec}(k))^{\oplus l_r}_\bbQ)(n)\,.
\end{equation}
Using the fact that $\mathfrak{h}(Y_L)_\bbQ$ is Kimura-finite, that $\wedge^2(\bbQ(1))=0$, and that Kimura-finiteness is stable under direct sums, direct summands, and tensor products, we hence conclude from \eqref{eq:directsum} that the Chow motive $\mathfrak{h}(X_L)_\bbQ$ is also Kimura-finite. The Kimura-finiteness of the (mixed) motive $M(X_L)_\bbQ$ follows now from the combination of the commutative diagram \eqref{eq:diagram2} with the fact that Kimura-finiteness is stable under duality and preserved by $\bbQ$-linear symmetric monoidal functors.

Let us now prove the equalities \eqref{eq:equality}-\eqref{eq:equality-new}. By definition, the even and odd Kimura-dimensions are invariant under duality. Therefore, thanks to the commutative diagram \eqref{eq:diagram2} and to the fact that the functor $\Psi$ is fully-faithful, it suffices to prove the following equalities:
\begin{eqnarray}
\mathrm{kim}_+(\mathfrak{h}(X_L)_{\bbQ,+}) &= &\mathrm{kim}_+(\mathfrak{h}(Y_L)_{\bbQ,+}) + (l_{\mathrm{dim}(L)}+ \cdots + l_{i-1})\,. \label{eq:equality1} \\
\mathrm{kim}_-(\mathfrak{h}(X_L)_{\bbQ,-}) & = & \mathrm{kim}_-(\mathfrak{h}(Y_L)_{\bbQ,-}) \label{eq:equality1-new}\,.
\end{eqnarray}
As explained in \cite[\S3]{Andre-survey}, we have the following equalities
\begin{eqnarray}\label{eq:equality-2}
& \mathrm{kim}_+(\mathfrak{h}(X_L)_{\bbQ,+})= \chi(\mathfrak{h}(X_L)_{\bbQ, +}) &  \mathrm{kim}_-(\mathfrak{h}(X_L)_{\bbQ,-})= - \chi(\mathfrak{h}(X_L)_{\bbQ, -})\,,
\end{eqnarray}
where $\chi$ stands for the Euler characteristic computed in the rigid symmetric monoidal category of Chow motives. In order to compute this Euler characteristic, consider the $\bbQ$-linear symmetric monoidal functor $HP^\pm \colon \NChow(k)_\bbQ \to \mathrm{Vect}_{\bbZ/2}(k)$, induced by periodic cyclic homology, with values in the category of finite dimensional $\bbZ/2$-graded $k$-vector spaces; see \cite[Thm.~7.2]{JEMS}. Note that every object $(V^+,V^-)$ of the category $\mathrm{Vect}_{\bbZ/2}(k)$ is Kimura-finite, that $(V^+,V^-)_+\simeq (V^+,0)$ and $(V^+,V^-)_-\simeq (0,V^-)$, and that $\chi((V^+,0))=\mathrm{dim}(V^+)$ and $\chi((0,V^-))=-\mathrm{dim}(V^-)$. Consider also the following composition
$$ \theta^\pm\colon \Chow(k)_\bbQ \stackrel{\pi}{\too} \Chow(k)_\bbQ/_{\!-\otimes \bbQ(1)} \stackrel{\Phi}{\too} \NChow(k)_\bbQ \stackrel{HP^\pm}{\too} \mathrm{Vect}_{\bbZ/2}(k)\,.$$
When restricted to the $\bbQ$-algebra of endomorphisms of the $\otimes$-unit, the functors $\pi$ and $\Psi$ become fully-faithful, and the functor $HP^\pm$ faithful. This implies that the Euler characteristic of any Chow motive can be computed after application of the functor $\theta^\pm$. Moreover, since the decomposition of a Kimura-finite object into even/odd parts is unique (see \cite[Prop.~6.3]{Kimura}), we hence conclude that
$$ \chi(\mathfrak{h}(X_L)_{\bbQ,+}) = \chi(\theta^\pm(\mathfrak{h}(X_L)_{\bbQ,+}))= \chi(\theta^\pm(\mathfrak{h}(X_L)_\bbQ)_+)=\mathrm{dim}(\theta^+(\mathfrak{h}(X_L)_\bbQ))$$
$$ \chi(\mathfrak{h}(X_L)_{\bbQ,-}) = \chi(\theta^\pm(\mathfrak{h}(X_L)_{\bbQ,-}))= \chi(\theta^\pm(\mathfrak{h}(X_L)_\bbQ)_-)=-\mathrm{dim}(\theta^-(\mathfrak{h}(X_L)_\bbQ))\,.$$
Thanks to these computations, the above equalities \eqref{eq:equality-2} reduces to 
\begin{eqnarray*}\label{eq:key1}
\mathrm{kim}_+(\mathfrak{h}(X_L)_{\bbQ,+}) = \mathrm{dim}(\theta^+(\mathfrak{h}(X_L)_\bbQ)) &&  \mathrm{kim}_-(\mathfrak{h}(X_L)_{\bbQ,-}) = \mathrm{dim}(\theta^-(\mathfrak{h}(X_L)_\bbQ)) \,.
\end{eqnarray*}
The above arguments hold {\em mutatis mutandis} for the Kimura-finite Chow motive $\mathfrak{h}(Y_L)_\bbQ$. Therefore, we also have the equality
\begin{eqnarray*}\label{eq:key2}
\mathrm{kim}_+(\mathfrak{h}(Y_L)_{\bbQ,+}) = \mathrm{dim}(\theta^+(\mathfrak{h}(Y_L)_\bbQ)) &&  \mathrm{kim}_-(\mathfrak{h}(Y_L)_{\bbQ,-}) = \mathrm{dim}(\theta^-(\mathfrak{h}(Y_L)_\bbQ)) \,.
\end{eqnarray*}
Now, by combining the above decomposition \eqref{eq:iso} with the commutative diagram \eqref{eq:diagram1}, we conclude that $\theta^\pm(\mathfrak{h}(X_L)_\bbQ)$ is isomorphic to the direct sum of $\theta^\pm(\mathfrak{h}(Y_L)_\bbQ)$ with $\oplus_{r=\mathrm{dim}(L)}^{i-1} HP^\pm(U(k)_\bbQ)^{\oplus l_r}$. Since $HP^\pm(U(k)_\bbQ)\simeq (k,0)$, this implies that 
\begin{eqnarray*}
\mathrm{dim}(\theta^+(\mathfrak{h}(X_L)_\bbQ)) & =& \mathrm{dim}(\theta^+(\mathfrak{h}(Y_L)_\bbQ)) + (l_{\mathrm{dim}(L)}+ \cdots + l_{i-1}) \\
\mathrm{dim}(\theta^-(\mathfrak{h}(X_L)_\bbQ)) & = & \mathrm{dim}(\theta^-(\mathfrak{h}(Y_L)_\bbQ))\,.
\end{eqnarray*}
The searched equalities \eqref{eq:equality1}-\eqref{eq:equality1-new} follow now automatically from the combination of the preceding six equalities. This concludes the proof of Theorem \ref{thm:2}.

\medbreak\noindent\textbf{Acknowledgments:} The author is grateful to Joseph Ayoub for an useful e-mail exchange concerning Schur-finiteness in the setting of Voevodsky's mixed motives.

\end{document}

\end{proof}